\documentclass{siamltex}
\usepackage{amsfonts}
\usepackage{graphicx}
\usepackage{epstopdf}
\usepackage{calligra}
\usepackage{algorithm}
\usepackage{algpseudocode}

\ifpdf
  \DeclareGraphicsExtensions{.eps,.pdf,.png,.jpg}
\else
  \DeclareGraphicsExtensions{.eps}
\fi

\usepackage{amsfonts,amsmath,amssymb}
\usepackage{mathrsfs,mathtools,stmaryrd}
\usepackage{enumerate}
\usepackage{xspace} 
\usepackage{esint}
\usepackage{graphicx}
\DeclareMathAlphabet{\mathpzc}{OT1}{pzc}{m}{it}

\DeclareMathOperator*{\argmin}{argmin}

\newtheorem{remark}[theorem]{Remark}

\newcommand{\R}{\mathbb{R}}

\newcommand{\Y}{\mathpzc{Y}}

\newcommand{\N}{\mathpzc{N}}

\newcommand{\C}{\mathcal{C}}
\newcommand{\calP}{\mathcal{P}}
\newcommand{\I}{\mathcal{I}}

\newcommand{\V}{\mathbb{V}}
\newcommand{\polV}{\mathbb{V}}
\newcommand{\polP}{\mathbb{P}}

\newcommand{\T}{\mathscr{T}}

\newcommand{\Ss}{\mathscr{S}}

\newcommand{\HL}{ \mbox{ \raisebox{7.0pt} {\tiny$\circ$} \kern-10.7pt} {H_L^1} }
\newcommand{\Wp}{ \mbox{ \raisebox{7.7pt} {\scriptsize$\circ$} \kern-10.1pt} {W^{1,p}} }
\newcommand{\Wpp}{ \mbox{ \raisebox{7.7pt} {\scriptsize$\circ$} \kern-10.1pt} {W^{1,p'}} }
\newcommand{\Sz}{ \mbox{ \raisebox{7.5pt} {\scriptsize$\circ$} \kern-10.1pt} {\Ss} }
\newcommand{\HLnew}{ \mbox{ \raisebox{7pt} {\scriptsize$\circ$} \kern-10.1pt}{H}^1_L }
\newcommand{\HLn}{{\mbox{\,\raisebox{5.1pt} {\tiny$\circ$} \kern-9.3pt}{H}^1_L  }}
\newcommand{\HLs}{{\mbox{\raisebox{8.7pt} {\scriptsize$\circ$} \kern-10.1pt}{H}^1_L  }}

\newcommand{\usf}{\mathsf{u}}

\newcommand{\fsf}{\mathsf{f}}

\newcommand{\ousf}{\bar{\mathsf{u}}}

\newcommand{\wsf}{\mathsf{w}}

\newcommand{\ue}{\mathscr{U}}

\DeclareMathOperator*{\tr}{tr_\Omega}

\newcommand{\boxednumber}[1]{\expandafter\readdigit\the\numexpr#1\relax\relax}

\newcommand{\Hsd}{\mathbb{H}^{-s}(\Omega)}

\newcommand{\Laps}{(-\Delta)^s}

\newcommand{\GRAD}{\nabla}
\newcommand{\DIV}{\textrm{div}}
\newcommand{\diff}{\, \mbox{\rm d}}
\newcommand{\ie}{i.e.,\@\xspace}

\newcommand{\Hs}{\mathbb{H}^s(\Omega)}

\DeclareMathOperator*{\diam}{diam}

\newcommand{\Nin}{\,{\mbox{\,\raisebox{6.0pt} {\tiny$\circ$} \kern-10.9pt}\N }}

\newtheorem{assumption}[theorem]{Assumption}

\title{{Optimization with respect to order in a fractional diffusion model: analysis, approximation and algorithmic aspects}\thanks{HA has been supported in part by NSF grant DMS-1521590. EO has been supported in part by CONICYT through FONDECYT project 3160201. AJS has been supported in part by NSF grant DMS-1418784.}}

\author{
  Harbir Antil\thanks{Department of Mathematical Sciences, George Mason University, Fairfax, VA 22030, USA.
    \texttt{hantil@gmu.edu}}
  \and
  Enrique Ot\'arola\thanks{Departamento de Matem\'atica, Universidad T\'ecnica Federico Santa Mar\'ia, Valpara\'iso, Chile.
    \texttt{enrique.otarola@usm.cl}}
  \and
  Abner J.~Salgado\thanks{Department of Mathematics, University of Tennessee, Knoxville, TN 37996, USA.
    \texttt{asalgad1@utk.edu}}
}

\date{Draft version of \today.}

\begin{document}

\maketitle

\begin{abstract}
We consider an identification problem, where the state $\usf$ is governed by a fractional elliptic equation and the unknown variable corresponds to the order $s \in (0,1)$ of the underlying operator.
We study the existence of an optimal pair $(\bar s, \bar\usf)$ and provide sufficient conditions for its local uniqueness. We develop semi-discrete and fully discrete algorithms to approximate the solutions to our identification problem and provide a convergence analysis.
We present numerical illustrations that confirm and extend our theory.
\end{abstract}

\begin{keywords}
optimal control problems, identification problems, fractional diffusion, bisection algorithm, finite elements, stability, fully--discrete methods, convergence.
\end{keywords}

\begin{AMS}
26A33,    
35J70,    
49J20,    
49K21,    
49M25,    
65M12,    
65M15,    
65M60.    
\end{AMS}


\section{Introduction}
\label{sec:intro}

Supported by the claim that they seem to better describe many processes; nonlocal models have recently become of great interest in the applied sciences and engineering. This is specially the case when long range (\ie nonlocal) interactions are to be taken into consideration; we refer the reader to \cite{MR3429730} for a far from exhaustive list of examples where such phenomena take place. However, the actual range and scaling laws of these interactions --- which determines the order of the model--- cannot always be directly determined from physical considerations. This is in stark contrast with models governed by partial differential equations (PDEs), which usually arise from a conservation law. This justifies the need to, on the basis of physical observations, identify the order of a fractional model.

In \cite{SprekelsV}, for the first time, this problem was addressed. The authors studied the optimization with respect to the order of the spatial operator in a nonlocal evolution equation; existence of solutions as well as first and second order optimality conditions were addressed. The present work is a natural extension of these results under the stationary regime: we address the local uniqueness of minimizers and propose a numerical algorithm to approximate them. In addition, we  study the convergence rates of our method.

To make matters precise, let $\Omega$ be an open and bounded domain in $\R^n$ ($n \geq 1$) with Lipschitz boundary $\partial \Omega$. Given a desired state $\usf_d : \Omega \rightarrow \mathbb{R}$ (the observations), we define the cost functional 
\begin{equation}
\label{def:J}
  J(s,\usf) = \frac{1}{2} \| \usf - \usf_d \|^2_{L^2(\Omega)} + \varphi(s),
\end{equation}
where, for some $a$ and $b$ satisfying that $0 \leq a<b \leq 1$, $s \in (a,b)$ and, $\varphi \in C^2(a,b)$ denotes a nonnegative convex function that satisfies
\begin{equation}
 \label{eq:cond_varphi}
 \lim_{s \downarrow a} \varphi(s) = + \infty =  \lim_{s \uparrow b}  \varphi(s).
\end{equation}
Examples of functions with these properties are
\[
  \varphi(s) = \frac{1}{(s-a)(b-s)}, \qquad \varphi(s) = \frac{ e^{\frac{1}{(b-s)}} }{s-a}.
\]

We shall thus be interested in the following identification problem: Find $(\bar s,\bar \usf)$ such that
\begin{equation}
\label{eq:minJ}
 J(\bar s,\bar \usf) = \min J(s,\usf)
\end{equation}
subject to the fractional state equation
\begin{equation}
\label{eq:laps}
\Laps \usf = \fsf \textrm{ in } \Omega, 
\end{equation}
where $\Laps$ denotes a fractional power of the Dirichlet Laplace operator $-\Delta$.
We immediately remark that, with no modification, our approach can be extended to problems where the state equation is $L^s \usf = \fsf$, where $L \wsf=-\DIV(A\GRAD \wsf)$, supplemented with homogeneous Dirichlet boundary conditions, as long as the diffusion coefficient $A$ is fixed, bounded and symmetric. In principle, one could also consider optimization with respect to order $s$ and the diffusion $A$, as this could accommodate for anisotropies in the diffusion process. We refer the reader to \cite{MR2997232}, and the references therein, for the case when $s=1$ is fixed and the optimization is carried out with respect to $A$.

We now comment on the choice of $a $ and $b$. The practical situation can be envisioned as the following: from measurements or physical considerations we have an expected range for the order of the operator, and we want to optimize within that range to best fit the observations. From the existence and optimality conditions point of view, there is no limitation on their values, as long as $0 \leq a < b \leq 1$. However, when we discuss the convergence of numerical algorithms, many of the estimates and arguments that we shall make blow up as $s \downarrow 0$ or $s \uparrow 1$ so we shall assume that $a > 0$ and $b < 1$.  How to treat numerically the full range of $s$ is currently under investigation.

Our presentation is organized as follows.  The notation and functional setting is introduced in section \ref{sec:notation}, where we also briefly describe, in section \ref{subsec:fractional_laplacian}, the definition of the fractional Laplacian. In section \ref{sec:fractional_identification}, we study the fractional identification problem \eqref{eq:minJ}--\eqref{eq:laps}. We analyze the differentiability properties of the associated control to state map (section \ref{subsec:control_to_state}) and derive existence results as well as first and second order optimality conditions and a local uniqueness result (section \ref{subsec:optimal_cond}). Section \ref{sec:Numerics} is dedicated to the design and analysis of a numerical algorithm to approximate the solution to \eqref{eq:minJ}--\eqref{eq:laps}. Finally, in section~\ref{s:nex} we illustrate the performance of our algorithm on several examples.

\section{Notation and preliminaries}
\label{sec:notation}
Throughout this work $\Omega$ is an open, bounded and convex polytopal subset of $\R^{n}$ $(n \geq 1)$ with boundary $\partial \Omega$. The relation $X \lesssim Y$ indicates that $X \leq CY$, with a nonessential constant $C$ that might change at each occurrence.

\subsection{The fractional Laplacian}
\label{subsec:fractional_laplacian}
Spectral theory for the operator $-\Delta$ yields the existence of a countable collection of eigenpairs $\{ \lambda_k, \varphi_k \}_{k \in \mathbb{N}} \subset \R^{+} \times H_0^1(\Omega)$ such that $\{ \varphi_k \}_{k \in \mathbb{N}}$ is an orthonormal basis of $L^2(\Omega)$ and an orthogonal basis of $H_0^1(\Omega)$ and
\begin{equation}
 \label{eq:spectral}
  -\Delta \varphi_k = \lambda_k \varphi_k \textrm{ in } \Omega,
 \qquad
 \varphi_k = 0 \textrm{ on } \partial \Omega,
 \qquad
 k \in \mathbb{N}.
\end{equation}
With this spectral decomposition at hand, we define the fractional powers of the Dirichlet Laplace operator, which for convenience we simply call the fractional Laplacian, as follows: For any $s \in (0,1)$ and $w \in C_0^{\infty}(\Omega)$,
\begin{equation}\label{eq:fracLap}
 (-\Delta)^s w := \sum_{k \in \mathbb{N}}  \lambda_k^s w_k \varphi_k, \quad w_k = (w,\varphi_k)_{L^2(\Omega)}:= \int_{\Omega} w \varphi_k \diff x.
\end{equation}
By density, this definition can be extended to the space
\begin{equation}
 \Hs = \left\{ w = \sum_{k \in \mathbb{N}} w_k \varphi_k \in L^2(\Omega): \sum_{k \in \mathbb{N}} \lambda_k^s w_k^2  < \infty \right\},
\end{equation}
which we endow with the norm
\begin{equation}
 \label{eq:Hsnorm}
 \| w \|_{\Hs} = \left( \sum_{k \in \mathbb{N}} \lambda_k^s w_k^2  \right)^{\frac{1}{2}};
\end{equation}
see \cite{MR3489634,CDDS:11,NOS} for details. The space $\Hs$ coincides with $[ L^2(\Omega), H_0^1(\Omega) ]_{s}$, \ie the interpolation space between $L^2(\Omega)$ and $H_0^1(\Omega)$; see \cite[Chapter 7]{Adams}. For $s \in (0,1)$, we denote by $\Hsd$ the dual space to $\Hs$ and remark that it admits the following characterization:
\begin{equation}
\label{eq:Hsd}
 \Hsd = \left\{ w = \sum_{k \in \mathbb{N} } w_k \varphi_k \in  \mathcal{D}'(\Omega)
 :  \sum_{k \in \mathbb{N}} \lambda_k^{-s} w_k^2  < \infty \right\},
\end{equation}
where $\mathcal{D}'(\Omega)$ denotes the space of distributions on $\Omega$. Finally, we denote by $\langle \cdot,\cdot \rangle$ the duality pairing between $\Hs$ and $\Hsd$.

\section{The fractional identification problem}
\label{sec:fractional_identification}

In this section we study the existence of minimizers for the fractional identification problem \eqref{eq:minJ}--\eqref{eq:laps}, as well as optimality conditions. We begin by introducing the so-called control to state map associated with problem \eqref{eq:minJ}--\eqref{eq:laps} and studying its differentiability properties. This will allow us to derive first order necessary and second order sufficient optimality conditions for our identification problem, as well as existence results.

\subsection{The control to state map}
\label{subsec:control_to_state}
In this subsection we study the differentiability properties of the control to state map $\mathcal{S}$ associated with \eqref{eq:minJ}--\eqref{eq:laps}, which we define as follows: Given a control $s \in (0,1)$, the map $\mathcal{S}$ associates to it the state $\usf = \usf(s)$ that solves problem \eqref{eq:laps} with the forcing term $\fsf \in \Hsd$. In other words, 
\begin{equation}
 \mathcal{S}:(0,1) \rightarrow \mathbb{H}^s(\Omega), \qquad s \mapsto \mathcal{S}(s) = \sum_{k \in \mathbb{N}}  \lambda_k^{-s} \fsf_k \varphi_k, 
 \label{eq:ctsmap}
\end{equation}
where $\fsf_k =  \langle \fsf,\varphi_k \rangle$ and $\{\lambda_k, \varphi_k \}_{k \in \mathbb{N}}$ are defined by \eqref{eq:spectral}. Since $\fsf \in \Hsd$, the characterization of the space $\Hsd$, given in \eqref{eq:Hsd}, allows us to immediately conclude that the map $\mathcal{S}$ is well--defined; see also 
\cite[Lemma 2.2]{CDDS:11}. 

Before embarking on the study of the smoothness properties of the map $\mathcal{S}$ we define, for $\lambda > 0$, the function $E_\lambda : (0,1) \to \R^+$ by
\begin{equation}
\label{eq:E}
 E_{\lambda}(s) = \lambda^{-s}, \qquad s \in (0,1).
\end{equation}
A trivial computation reveals that 
\begin{equation}
\label{eq:Dsl-s}
  D_s^m E_{\lambda}(s) = (-1)^m \ln^m( \lambda ) E_{\lambda}(s), \quad m \in \mathbb{N},
\end{equation}
from which immediately follows 
that, for $m \in \mathbb{N}$, we have the estimate
\begin{equation}
\label{eq:basic_estimate}
 |D_s^m E_{\lambda}(s)| \lesssim s^{-m},
\end{equation}
where the hidden constant is independent of $s$, it remains bounded as $\lambda \uparrow \infty$, but blows up as $\lambda  \downarrow 0$; compare with \cite[eq.~(2.27)]{SprekelsV}.

With this auxiliary function at hand we proceed, following  \cite{SprekelsV}, to study the differentiability properties of the map $\mathcal{S}$. To begin we notice the inclusion $\mathcal{S}((0,1)) \subset L^2(\Omega)$ so we consider $\mathcal{S}$ as a map with range in $L^2(\Omega)$ and we will denote by $\interleave \cdot \interleave$ the norm of $\mathcal L(\R,L^2(\Omega))$.

\begin{theorem}[properties of $\mathcal S$]
\label{thm:Ds_and_Dss}
Let $\mathcal{S}: (0,1) \rightarrow L^2(\Omega)$ be the control to state map, defined in \eqref{eq:ctsmap}, and assume that $\fsf \in L^2(\Omega)$.  For every $s \in (0,1)$ we have that
\begin{equation}
\label{eq:Sbbindep}
  \| \mathcal{S}(s) \|_{L^2(\Omega)} \lesssim 1,
\end{equation}
where the hidden constant depends on $\Omega$ and $\| \fsf \|_{L^2(\Omega)}$, but not on $s$. In addition, $\mathcal{S}$ is three times Fr\'echet differentiable; the first and second derivatives of $\mathcal{S}$ are characterized as follows: for $h_1, h_2 \in \mathbb{R}$, we have that 
\begin{equation}
\label{eq:DS}
D_s \mathcal{S}(s)[h_1] = h_1 D_s \usf(s),
\qquad
D_s^2 \mathcal{S}(s)[h_1,h_2] = h_1 h_2 D_s^2 \usf(s),
\end{equation}
where
\[
 D_s \usf(s) = -\sum_{k \in \mathbb{N} } \lambda_k^{-s} \ln(\lambda_k)  \fsf_k \varphi_k, \qquad
 D_s^2 \usf(s) = \sum_{k \in \mathbb{N} } \lambda_k^{-s} \ln^2(\lambda_k)  \fsf_k \varphi_k.
\]
Finally, for $m=1,2,3$, we have
\begin{equation}
\label{eq:Dscontrol}
\interleave D_s^m \mathcal{S}(s) \interleave \lesssim s^{-m},
\end{equation}
where the hidden constants are independent of $s$.
\end{theorem}
\begin{proof} 
Let $s \in (0,1)$. To shorten notation we set $\usf = \mathcal{S}(s)$. Using \eqref{eq:ctsmap} we have that
\begin{equation}
\label{eq:u_f_L2}
  \|\usf \|_{L^2(\Omega)}^2 = \sum_{k \in \mathbb{N}} \lambda_k^{-2s} \fsf_k^2 \leq \lambda_1^{-2s} \| \fsf \|_{L^2(\Omega)}^2,
\end{equation}
where we used that, for all $k \in \mathbb{N}$, $0 < \lambda_1 \leq \lambda_k$. Since $\sup_{s \in [0,1]} \lambda_1^{-2s}$ is bounded, we obtain \eqref{eq:Sbbindep}.

We now define, for $N \in \mathbb{N}$, the partial sum $w_N = \sum_{k = 1}^{N} \lambda_k^{-s} \fsf_k \varphi_k$. Evidently, as $N \to \infty$, we have that $w_N \to \usf$ in $L^2(\Omega)$. Moreover, differentiating with respect to $s$ we immediately obtain, in light of \eqref{eq:Dsl-s}, the expression
\[
  D_s w_N = -\sum_{k \leq N} \lambda_k^{-s} \ln (\lambda_k)  \fsf_k \varphi_k,
\]
and, using \eqref{eq:Dsl-s} and \eqref{eq:basic_estimate}, that
\[
  \|D_s w_N \|^2 _{L^2(\Omega)}  = \sum_{k \leq N} |D_s E_{\lambda_k}(s)|^2  \fsf_k^2 \lesssim \frac1{s^2} \| \fsf \|_{L^2(\Omega)}^2,
\]
where we used, again, that the eigenvalues are strictly away from zero. This estimate allows us to conclude that, as $N \to \infty$, we have $D_s w_N \to D_s \usf$ in $L^2(\Omega)$ and the bound
\begin{equation}
\label{eq:Dsu}
    \| D_s \usf(s) \|_{L^2(\Omega)} \lesssim s^{-1} \| \fsf \|_{L^2(\Omega)}.
\end{equation}

Let us now prove that $\mathcal{S}$ is Fr\'echet differentiable and that \eqref{eq:DS} holds. Taylor's theorem, in conjunction with \eqref{eq:Dsl-s}, yields that, for every $l \in \mathbb{N}$ and $h_1 \in \R$, we have
\[
e_{l,s} := | E_{\lambda_l}(s + h_1) - E_{\lambda_l}(s) - D_s E_{\lambda_l}(s) h_1 | = \frac{1}{2} h_1^2|D_s^2 E_{\lambda_l}(\theta)|,
\]
for some $\theta \in (s -|h_1|,s+|h_1|)$. Now, if $|h_1| < s/2$, we have that $\theta^{-2} < 4 s^{-2}$, and thus, in view of estimate \eqref{eq:basic_estimate}, that
\[
e_{l,s} = \frac{1}{2} h_1^2|D_s^2 E_{\lambda_l}(\theta)| \lesssim h_1^2 s^{-2}.
\]
This last estimate allows us to write 
\[
  \| \mathcal{S}(s+h_1) - \mathcal{S}(s) - D_{s} \usf(s) h_1 \|^2_{L^2(\Omega)} = \sum_{k \in \mathbb{N}} e_{k,s}^2 \fsf_k^2 
  \lesssim h_1^4 s^{-4}  \| \fsf \|_{L^2(\Omega)}^2,
\]
where the hidden constant is independent of $h_1$ and $s$. The previous estimate shows that $\mathcal S: (0,1) \to L^2(\Omega)$ is Fr\'echet differentiable and that $D_s \mathcal{S}(s)[h_1] = h_1 D_s \usf(s)$. Finally, using \eqref{eq:Dsu}, we conclude, estimate \eqref{eq:Dscontrol} for $m=1$.

Similar arguments can be applied to show the higher order Fr\'echet differentiability of $\mathcal S$ and to derive estimate \eqref{eq:Dscontrol} for $m=2,3$. For brevity, we skip the details.
\end{proof}

\subsection{Existence and optimality conditions}
\label{subsec:optimal_cond}

We now proceed to study the existence of a solution to problem \eqref{eq:minJ}--\eqref{eq:laps} as well as to characterize it via first and second order optimality conditions. We begin by defining the reduced cost functional
\begin{equation}
 f(s) = J(s,\mathcal{S}(s)),
 \label{def:f}
\end{equation}
where $\mathcal{S}$ denotes the control to state map defined in \eqref{eq:ctsmap} and $J$ is defined as in \eqref{def:J}; we recall that $\varphi \in C^2(a,b)$. Notice that, owing to Theorem~\ref{thm:Ds_and_Dss}, $\mathcal S$ is three times Fr\'echet differentiable. Consequently, $f \in C^2(a,b)$ and, moreover, it verifies conditions similar to \eqref{eq:cond_varphi}. These properties will allow us to show existence of an optimal control. We begin with a definition.

\begin{definition}[optimal pair]
The pair $(\bar{s},\bar{\usf}(\bar{s})) \in (a,b) \times \mathbb{H}^{\bar{s}}(\Omega)$ is called 
optimal for problem \eqref{eq:minJ}--\eqref{eq:laps} if $\bar{\usf}(\bar{s}) = \mathcal{S}(\bar{s})$ and 
\[
 f(\bar{s}) \le f(s),
\]
for all $(s,\usf(s)) \in (a,b) \times \mathbb{H}^s(\Omega)$ such that $\usf(s) = \mathcal{S}(s)$.
\end{definition}

\begin{theorem}[existence]
\label{thm:existence}
There is an optimal pair $(\bar{s},\bar{\usf}(\bar{s})) \in (a,b) \times \mathbb{H}^{\bar{s}}(\Omega)$ for problem \eqref{eq:minJ}--\eqref{eq:laps}.
\end{theorem}
\begin{proof}
Let $\{a_l\}_{l \in \mathbb{N}}, \{b_l\}_{l \in \mathbb{N}} \subset (a,b)$ be such that, for every $l \in \mathbb N$, $a<a_{l+1}<a_l < b_l<b_{l+1}<b$ and $a_l \to a$, $b_l \to b$ as $l \to \infty$. Denote $I_l = [a_l,b_l]$ and consider the problem of finding
\[
  s_l = \argmin_{s \in I_l} f(s).
\]
The properties of $f$ guarantee its existence. Notice that, since the intervals $I_l$ are nested, we have
\[
  f(s_m) \leq f(s_l), \qquad m \geq l.
\]
We have thus constructed a sequence $\{s_l\}_{l \in \mathbb{N}} \subset (a,b)$ from which we can extract a convergent subsequence, which we still denote by the same $\{s_l\}_{l \in \mathbb{N}}$, such that $s_l \to \bar s \in [a,b]$. We claim that $f$ attains its infimum, over $(a,b)$, at the point $\bar s$.

Let us begin by showing that, in fact, $\bar s \in (a,b)$. The decreasing property of $\{f(s_l)\}_{l \in \mathbb{N}}$ shows that
\[
  f(\bar s) \leq f(s_l), \quad \forall l \in \mathbb{N},
\]
which, if $\bar s = a$ or $\bar s=b$, would lead to a contradiction with the fact that $f(s) \geq \varphi(s)$ and \eqref{eq:cond_varphi}.

Let $s_\star$ be any point of $(a,b)$. The construction of the intervals $I_l$ guarantee that there is $L \in \mathbb{N}$ for which $s_\star \in I_l$ whenever $l > L$. Therefore, we have
\[
  f(\bar s) \leq f(s_l) = \min_{s \in I_l} f(s) \leq f(s_\star).
\]
Which shows that $\bar s$ is a minimizer.

Since $\mathcal S$, as a map from $(a,b)$ to $L^2(\Omega)$, is continuous --- even differentiable --- we see that there is $\bar \usf \in L^2(\Omega)$, for which $\mathcal S(s_l) \to \bar \usf $ in $L^2(\Omega)$ as $l \to \infty$. Let us now show that, indeed, $\bar \usf \in \mathbb{H}^{\bar s}(\Omega)$ and that it satisfies the state equation.

Set $\bar\usf = \sum_{k\in \mathbb N} \bar\usf_k \varphi_k$ and notice that, as $l \to \infty$,
\[
  ( \mathcal{S}(s_l) - \bar \usf, \varphi_m )_{L^2(\Omega)} = \lambda_m^{-s_l} \fsf_m - \bar\usf_m \to \lambda_m^{-\bar s}\fsf_m -\bar\usf_m.
\]
Therefore $\bar\usf_m = \lambda_m^{-\bar s} \fsf_m$. This shows that $\bar\usf \in \mathbb{H}^{\bar s}(\Omega)$ and that $\bar\usf$ solves \eqref{eq:laps}.

The result is thus proved.
\end{proof}


We now provide first order necessary and second order sufficient optimality conditions for the identification problem \eqref{eq:minJ}--\eqref{eq:laps}.

\begin{theorem}[optimality conditions]
\label{thm:foopt}
Let $(\bar{s}, \bar\usf)$ be an optimal pair for problem \eqref{eq:minJ}--\eqref{eq:laps}. Then it satisfies the following first order necessary optimality condition
\begin{equation}
\label{eq:first_order}
(\bar{\usf} - \usf_d, D_s \bar{\usf})_{L^2(\Omega)} + \varphi'(\bar s) = 0.
\end{equation}
On the other hand, if $(\bar s, \bar \usf)$, with  $\bar \usf = \mathcal{S}(\bar s)$, satisfies \eqref{eq:first_order} and, in addition, the second order optimality condition
\begin{equation}
\label{eq:second_order}
( D_s \bar{\usf}, D_s \bar{\usf})_{L^2(\Omega)} + ( \bar{\usf} - \usf_d, D_{s}^2 \bar{\usf})_{L^2(\Omega)} + \varphi''(\bar s) > 0
\end{equation}
holds, then $(\bar s, \bar \usf)$ is an optimal pair.
\end{theorem}
\begin{proof}
Since, as shown in Theorem~\ref{thm:existence}, $\bar s \in (a,b)$, the first order optimality condition reads:
 \begin{equation}
 \label{eq:fprima}
 f'(\bar s) = (\mathcal{S}(\bar s) - \usf_d, D_s \mathcal{S}(\bar s))_{L^2(\Omega)} + \varphi'(\bar s) = 0.
 \end{equation}
The characterization of the first order derivative of $\mathcal{S}$, given in Theorem \ref{thm:Ds_and_Dss}, allows us to conclude \eqref{eq:first_order}. A similar computation reveals that
\begin{equation}
\label{eq:fdosprima}
 f''(\bar s) = (D_s \mathcal{S}(\bar s), D_s \mathcal{S}(\bar s))_{L^2(\Omega)} + (\mathcal{S}(\bar s) - \usf_d, D_{s}^2 \mathcal{S}(\bar s))_{L^2(\Omega)} + \varphi''(\bar s).
\end{equation}
Using, again, the characterization for the first and second order derivatives of $\mathcal{S}$ given in Theorem \ref{thm:Ds_and_Dss} we obtain \eqref{eq:second_order}. This concludes the proof.
\end{proof}

Let us now provide a sufficient condition for local uniqueness of the optimal identification parameter $\bar s$. To accomplish this task we assume that the function $\varphi$, that defines the functional $J$ in \eqref{def:J}, is strongly convex with parameter $\xi$, \ie for all points $s_1, s_2$ in $(a,b)$, we have that
\begin{equation}
  \label{eq:strictly_convex}
  \left( \varphi'(s_1) - \varphi'(s_2) \right) \cdot (s_1-s_2) \geq \xi | s_1 - s_2 |^2.
\end{equation}
We thus present the following result.

\begin{lemma}[second--order sufficient conditions]
Let $\bar s$ be optimal for problem \eqref{eq:minJ}--\eqref{eq:laps} and $f$ be defined as in \eqref{def:f}. If $\| \fsf \|_{L^2(\Omega)}$ and $\| \usf_d \|_{L^2(\Omega)}$ are sufficiently small, then there exist a constant $\vartheta > 0$ such that
 \begin{equation}
 \label{eq:d2fleq2sigma}
  f''( \bar s ) \geq \vartheta.
 \end{equation}
  \end{lemma}
\begin{proof}
On the basis of \eqref{eq:fdosprima}, we invoke the strong convexity of $\varphi$ to conclude that
\[
 f''( \bar s ) \geq \| D_s \mathcal{S}(\bar s) \|^2_{L^2(\Omega)} + (\mathcal{S}(\bar s) - \usf_d, D_{s}^2 \mathcal{S}( \bar s))_{L^2(\Omega)} + \xi.
\]
It thus suffices to control the term $(\mathcal{S}(\bar s) - \usf_d, D_{s}^2 \mathcal{S}( \bar s))_{L^2(\Omega)}$; and to do so we use the estimates of Theorem \ref{thm:Ds_and_Dss}. In fact, we have that
\[
  |(\mathcal{S}(\bar s) - \usf_d, D_{s}^2 \mathcal{S}(\bar s))_{L^2(\Omega)}| \leq C_1 \left(C_2 \| \fsf \|_{L^2(\Omega)} + \|\usf_d \|_{L^2(\Omega)} \right ) \bar s^{-2} \| \fsf \|_{L^2(\Omega)}, 
\]
where $C_1$ and $C_2$ depend on $\Omega$ and the operator $-\Delta$ but are independent of $\bar s$, $\fsf$ and $\usf_d$. Since Theorem \ref{thm:existence} guarantees that $\bar s \in (a,b)$,  we conclude that the right hand side of the previous expression is bounded. This, in view of the fact that $\| \fsf \|_{L^2(\Omega)}$ and $\| \usf_d \|_{L^2(\Omega)}$ are sufficiently small, concludes the proof.
\end{proof}

As a consequence of the previous Lemma we derive, for the reduced cost functional $f$,  a convexity property that will be important to analyze the fully discrete scheme of section~\ref{sec:Numerics}, and a quadratic growth condition that implies the local uniqueness of $\bar s$.

\begin{corollary}[convexity and quadratic growth]
\label{co:local_convexity}
Let $\bar s$ be optimal for problem \eqref{eq:minJ}--\eqref{eq:laps} and $f$ be defined as in \eqref{def:f}. If $\| \fsf \|_{L^2(\Omega)}$ and $\| \usf_d \|_{L^2(\Omega)}$ are sufficiently small, then there exists $\delta > 0$ such that 
\begin{equation}
\label{eq:local_convexity}
  (f'(s) - f'(\bar s)) \cdot( s - \bar s) \geq \frac {\vartheta}{2} |s - \bar s|^2 \qquad \forall s \in (a,b) \cap (\bar s - \delta, \bar s + \delta),
\end{equation}
where $\vartheta$ is the constant that appears in \eqref{eq:d2fleq2sigma}. In addition, we have the quadratic growth condition
\begin{equation}
\label{eq:quadratic growth}
  f(s) \geq f(\bar s) + \frac{\vartheta}{4} |s - \bar s|^2 \qquad \forall s \in (a,b) \cap (\bar s - \delta, \bar s + \delta).
\end{equation}
In particular, $f$ has a local minimum at $\bar s$. Moreover, this minimum is unique in $(\bar s-\delta, \bar s + \delta) \cap (a,b)$.
\end{corollary}
\begin{proof}
Estimates \eqref{eq:local_convexity} and \eqref{eq:quadratic growth} follow immediately from an application of  Taylor's theorem and estimate \eqref{eq:d2fleq2sigma}; see \cite[Theorem 4.23]{Tbook} for details. The local uniqueness follows immediately from \eqref{eq:quadratic growth}.
\end{proof}

\section{A numerical scheme for the fractional identification problem}
\label{sec:Numerics}
In this section we propose a numerical method that approximates the solution to the fractional identification problem \eqref{eq:minJ}--\eqref{eq:laps}. To be able to provide a convergence analysis of the proposed method we make the following assumption.

\begin{assumption}[compact subinterval]
\label{ass:alphabeta}
The optimization bounds $a$ and $b$ satisfy
\[
  0 < a < b < 1.
\]
\end{assumption}

The scheme that we propose below is based on the discretization of the first order optimality condition \eqref{eq:first_order}: we discretize the first derivative $D_s \usf(s)$ in \eqref{eq:first_order} using a centered difference and then we approximate the solution to the state equation \eqref{eq:laps} with the finite element techniques introduced in \cite{NOS}.

\subsection{Discretization in $s$}
\label{subsec:bisection}
To set the ideas, we first propose a scheme that only discretizes the variable $s$ and analyze its convergence properties. We begin by introducing some terminology. Let $\sigma > 0$ and $s \in (a,b)$ such that $s \pm \sigma \in (a,b)$. We thus define, for $\psi : (a,b) \to \R$, the centered difference approximation of $D_s \psi$ at $s$ by
\begin{equation}
\label{eq:dsigma}
 d_{\sigma} \psi(s) := \frac{\psi(s+\sigma) - \psi(s-\sigma)}{2 \sigma}.
\end{equation}
If $\psi \in C^3(a,b)$, a basic application of Taylor's theorem immediately yields the estimate
\begin{equation}
  \label{eq:bouncds}
 \left| D_s \psi(s) - d_{\sigma} \psi(s) \right| \leq \frac{\sigma^2}{3}  \| D_s^3 \psi \|_{L^{\infty}(s-\sigma,s+\sigma)}.
\end{equation}
We also define the function $j_{\sigma}: (a,b) \rightarrow \mathbb{R}$ by
\begin{equation}
 \label{def:jsigma}
 j_{\sigma}(s) = (\usf(s) - \usf_d, d_{\sigma} \usf(s) )_{L^2(\Omega)} + \varphi'(s),
\end{equation}
where $\usf(s)$ denotes the solution to \eqref{eq:laps}. Finally, a point $s_\sigma \in (a,b)$ for which
\begin{equation}
\label{eq:ssigma}
 j_\sigma (s_\sigma) = 0,
\end{equation}
will serve as an approximation of the optimal parameter $\bar s$.

\begin{algorithm}
\caption{Bisection algorithm.}
\label{alg:Bisection}
  \begin{algorithmic}[1]
  \State $0 < \sigma \ll 1$ and set $s_l, s_r \in (a,b)$, with $s_l < s_r$.;   \Comment{Initialization}

  \Comment{We take care of possible degenerate cases}
  \If{$j_\sigma(s_l) = 0$}
    \State $s_\sigma = s_l$;
  \EndIf
  \If{$j_\sigma(s_r) = 0$}
    \State $s_\sigma = s_r$;
  \EndIf
  
  \Comment{Root isolation}
  \While{$j_\sigma(s_r)<0$}
    \State $s_r := s_r + \sigma$;
  \EndWhile
  \While{$j_\sigma(s_l)>0$}
    \State $s_l := s_l - \sigma$;
  \EndWhile
  
  \Comment{Bisection}
  \State $k=1$; \label{lin:banana}
  \Repeat
    \State $s_k = \frac12(s_l+s_r)$;
    \If{$j_\sigma(s_k)=0$}
      \State $s_\sigma = s_k$;
      \State \textbf{break}; \Comment{The solution has been found}
    \EndIf
    \If{$j_{\sigma}(s_l) j_{\sigma}( s_k ) > 0$} \Comment{Sign check}
      \State $s_l = s_k$;
    \Else
      \State $s_r = s_k$;
    \EndIf
    \State $k = k+1$;
  \Until{\textbf{forever}}
  \end{algorithmic}
\end{algorithm}

Notice that, in \eqref{def:jsigma}, the definition of $j_\sigma$ coincides with the first order optimality condition \eqref{eq:first_order}, when we replace the derivative of the state, \ie $D_s \usf$, by its centered difference approximation, as defined in \eqref{eq:dsigma}. The existence of $s_\sigma$ will be shown by proving convergence of Algorithm~\ref{alg:Bisection} which, essentially, is a bisection algorithm. In addition, if the algorithm reaches line \ref{lin:banana},
since $j_{\sigma} \in C([s_l, s_r])$ and it takes values of different signs at the endpoints, the intermediate value theorem guarantees that the bisection step will produce a sequence of values that we use to approximate the root of $j_\sigma$.
It remains then to show that we can eventually find the requisite interval $[s_l,s_r] \subset (a,b)$. This is the content of the following result.

\begin{lemma}[root isolation]
\label{lm:zeroone}
If $\sigma$ is sufficiently small, there exist $s_l$ and $s_r$ in $(a,b)$ such that $j_{\sigma}(s_l) < 0$ and $j_{\sigma}(s_r) > 0$, \ie the root isolation step in Algorithm~\ref{alg:Bisection} terminates.
\end{lemma}
\begin{proof}
We begin the proof by noticing that, for  $s \in (\sigma,1-\sigma) \subset (a,b)$, the estimates of Theorem~\ref{thm:Ds_and_Dss} immediately yield the existence of a constant $C>0$ such that
\begin{equation}
\label{eq:step_0}
 \left| \left( \usf(s) - \usf_d, d_{\sigma} \usf (s) \right)_{L^2(\Omega)} \right| \leq  \frac{C}{\sigma},
\end{equation}
where $C$ depends on $\Omega$, $\usf_d$ and $\fsf$ but not on $s$ or $\sigma$. 

On the other hand, since property \eqref{eq:cond_varphi} implies that $\varphi'(s) \rightarrow -\infty$ as $s \downarrow a$, we deduce the existence of $\epsilon_{l} > 0$ such that, if $s \in (a, a+\epsilon_{l})$ then $\varphi'(s) < -C/\sigma$. Assume that $\sigma < \epsilon_{l}$. Consequently, in view of the bound \eqref{eq:step_0}, definition \eqref{def:jsigma} immediately implies that, for every $s \in (a+\sigma,a+\epsilon_{l})$, we have the estimate
\[
j_{\sigma}(s)  \leq  \frac{C}{\sigma} + \varphi'(s) < 0.
\]
Similar arguments allow us to conclude the existence of $\epsilon_r >0$ such that, if $s \in (b-\epsilon_r, b)$ then $\varphi'(s) > C/\sigma$. Assume that $\sigma < \epsilon_r$ . We thus conclude that, for every $s \in (b-\epsilon_r,b-\sigma)$, we have the bound
\[
j_{\sigma}(s)  \geq - \frac{C}{\sigma} + \varphi'(s) > 0.
\]
%
%

In light of the previous estimates we thus conclude that, for $\sigma < \min\{\epsilon_l,\epsilon_r\}$, we can find $s_l$ and $s_r$ in $(a,b)$ such that $j_{\sigma}(s_l)<0$ and $j_{\sigma}(s_r)>0$. This concludes the proof.
\end{proof}

From Lemma~\ref{lm:zeroone} we immediately conclude that the bisection algorithm can be performed and exhibits the following convergence property.

\begin{lemma}[convergence rate: bisection method]
\label{lem:rate_bisection}
The sequence $\{ s_k \}_{k \geq 1}$ generated by the bisection algorithm satisfies
\begin{equation}
\label{eq:rate_bisection}
  |  s_{\sigma} - s_{k} | \lesssim 2^{-k}.
\end{equation}
In addition, there exists $s_l$ and $s_r$ such that $a<s_l<s_r < b$ and $s_{\sigma} \in (s_l,s_r)$.
\end{lemma}

The results of Lemmas \ref{lm:zeroone} and \ref{lem:rate_bisection} guarantee that, for a fixed $\sigma$, the bisection algorithm can be performed and exhibits a convergence rate dictated by \eqref{eq:rate_bisection}. Let us now discuss the convergence properties, as $\sigma \to 0$, of this semi-discrete method. We begin with two technical lemmas.

\begin{lemma}[convergence of $j_\sigma$]
\label{lem:jsigmatoj}
Let $j_\sigma : (a, b ) \to \R$ be defined as in \eqref{def:jsigma}, then, $j_\sigma \rightrightarrows f'$ on $(a,b)$ as $\sigma \to 0$.
\end{lemma}
\begin{proof}
From the definitions we obtain that, whenever $s \in (a,b)$
\begin{align*}
  |f'(s) - j_\sigma(s)| &= \left| (\usf(s) - \usf_d , D_s \usf(s) - d_\sigma \usf(s) )_{L^2(\Omega)} \right| \\
    &\lesssim \sup_{s \in [a,b]} \| D_s \usf(s) - d_\sigma \usf(s) \|_{L^2(\Omega)},
\end{align*}
where the hidden constant depends on $\usf_d$ and estimate \eqref{eq:Sbbindep}. Since, from Theorem~\ref{thm:Ds_and_Dss} we know that the control to state map is three times differentiable, we can conclude that
\[
  \| D_s \usf(s) - d_\sigma \usf(s) \|_{L^2(\Omega)} \lesssim \frac{\sigma^2}{a^3},
\]
where we used a formula analogous to \eqref{eq:bouncds} and estimate \eqref{eq:Dscontrol}. The fact that $a>0$ (Assumption~\ref{ass:alphabeta}) allows us to conclude.
\end{proof}

With the uniform convergence of $j_\sigma$ at hand, we can obtain the convergence of its roots to parameters that are optimal.

\begin{lemma}[convergence of $s_\sigma$]
\label{lem:ssigmatos}
The family $\{s_\sigma\}_{\sigma>0}$ contains a convergent subsequence. Moreover, the limit of any convergent subsequence satisfies \eqref{eq:first_order}.
\end{lemma}
\begin{proof}
The existence of a convergent subsequence follows from the fact that $\{s_\sigma\}_{\sigma>0} \subset [a,b]$. Moreover, as in Theorem~\ref{thm:existence}, we conclude that the limit is in $(a,b)$. Let us now show that any limit satisfies \eqref{eq:first_order}.

By Lemma~\ref{lem:jsigmatoj}, for any $\varepsilon>0$, if $\sigma$ is sufficiently small, we have that
\[
  |f'(s_\sigma)| = |f'(s_\sigma)-j_\sigma(s_\sigma)| < \varepsilon
\]
which implies that $f'(s_\sigma) \to 0$ as $\sigma \to 0$. Let now $\{s_{\sigma_k}\}_{k \in \mathbb N} \subset \{s_\sigma\}$ be a convergent subsequence. Denote the limit point by $\underline{s} \in (a,b)$. By continuity of $f'$ we have $ f'(s_{\sigma_k}) \to f'(\underline{s})$ which implies that
\[
  f'(\underline{s}) = 0.
\]
\end{proof}

\begin{remark}[stronger convergence]
\label{rem:notfullseq}
It is expected that we cannot prove more than convergence up to subsequences, since there might be more than one $s$ that satisfies \eqref{eq:first_order}. If there is a unique optimal $s$, then the previous result implies that the family $\{s_\sigma\}_{\sigma>0}$ converges to it.
\end{remark}

In what follows, to simplify notation, we denote by $\{s_\sigma\}_{\sigma>0}$ any convergent subfamily. The next result then provides a rate of convergence.

\begin{theorem}[convergence rate in $\sigma$]
\label{thm:rate_s}
Let $\bar s$ denote a solution to the identification problem \eqref{eq:minJ}--\eqref{eq:laps} and let $s_{\sigma}$ be its approximation defined as the solution to equation \eqref{eq:ssigma}. If $\sigma$ is sufficiently small then we have 
\[
  | \bar{s} - s_{\sigma} | \lesssim \frac{\sigma^2}{a^3} \left( \| \fsf \|_{L^2(\Omega)} + \| \usf_d \|_{L^2(\Omega)} \right),
\]
where the hidden constant is independent of $\bar s$, $s_{\sigma}$, $\sigma$, $\fsf$ and $\usf_d$.
\end{theorem}
\begin{proof}
We begin by considering the parameter $\sigma$ sufficiently small such that $s_{\sigma} \in (\bar s - \delta, \bar s + \delta)$, where $\delta > 0$ is as in the statement of Corollary \ref{co:local_convexity}. Thus, an application of the estimate \eqref{eq:local_convexity} in conjunction with the fact that $j_{\sigma}(s_{\sigma}) = 0$  allow us to conclude that
 \begin{align*}
  \tfrac{\vartheta}{2} | \bar s - s_{\sigma} |^2 & \leq ( f'(\bar s) - f'(s_{\sigma})) \cdot (\bar s - s_{\sigma}) = f'(s_{\sigma}) (s_{\sigma} - \bar s) 
  \\
  & = (f'(s_{\sigma}) - j_{\sigma} (s_{\sigma}))\cdot ( s_{\sigma} -\bar s).
 \end{align*}
Consequently, following Lemma~\ref{lem:jsigmatoj} we obtain that
\begin{equation}
\label{eq:auxsigma}
  \begin{aligned}
    \frac\vartheta2 | \bar s - s_{\sigma} | &\leq \left| \left(\usf(s_{\sigma}) - \usf_d, D_s \usf(s_{\sigma}) - d_{\sigma} \usf(s_{\sigma})\right)_{L^2(\Omega)}\right| \\ &\lesssim \frac{\sigma^2}{a^3} \left( \| \fsf \|_{L^2(\Omega)} + \| \usf_d \|_{L^2(\Omega)} \right).
  \end{aligned}
\end{equation}
The theorem is thus proved.
\end{proof}

\subsection{Space discretization}
The goal of this subsection is to propose, on the basis of the bisection algorithm of section \ref{subsec:bisection}, a fully discrete scheme that approximates the solution to problem \eqref{eq:minJ}--\eqref{eq:laps}. To accomplish this task we will utilize the discretization techniques introduced in \cite{NOS} that provides an approximation to the solution to the  fractional diffusion problem \eqref{eq:laps}. In order to make the exposition as clear as possible, we briefly review these aforementioned techniques below.

\subsubsection{A discretization technique for fractional diffusion}
\label{subsec:NOS}
Exploiting the cylindrical extension proposed and investigated in \cite{CT:10,CDDS:11,ST:10}, that is in turn inspired in the breakthrough by L. Caffarelli and L. Silvestre analyzed in \cite{CS:07}, the authors of \cite{NOS} have proposed a numerical technique to approximate the solution to problem \eqref{eq:laps} that is based on an anisotropic finite element discretization of the following local and nonuniformly elliptic PDE:
\begin{equation}
\label{eq:extension}
 \textrm{div}(y^{\alpha} \nabla \ue ) = 0 \textrm{ in } \C, \qquad \ue = 0 \textrm{ on } \partial_L \C, \qquad  \partial_{\nu^\alpha} \ue = d_s \fsf \textrm{ in } \Omega.
\end{equation}
Here, $\C$ denotes the semi--infinite cylinder with base $\Omega$ defined by
\[
 \C = \Omega \times (0,\infty) \subset \R_+^{n+1} = \{ (x',y): x' \in \R^n, y > 0 \},
\]
and $\partial_L \C = \partial \Omega \times [0,\infty)$ its lateral boundary. In addition, $d_s = 2^\alpha\Gamma(1-s)/\Gamma(s)$ and
\[
\partial_{\nu^\alpha} \ue = - \lim_{y \rightarrow 0+} y^\alpha \ue_{y}.
\]
Finally, $\alpha = 1-2s \in (-1,1)$. Although degenerate or singular, the variable coefficient $y^\alpha$  satisfies a key property. Namely, it belongs to the Muckenhoupt class $A_2(\R^{n+1})$. This allows for an optimal piecewise polynomial interpolation theory \cite{NOS}.

To state the results of \cite{CT:10,CS:07,CDDS:11,ST:10}, we define the weighted Sobolev space 
\begin{equation*}
\label{eq:defofHL}
  \HL(y^\alpha,\C) = \left\{ w \in H^1(y^\alpha,\C): w = 0 \textrm{ on } \partial_L \C\right\},
\end{equation*}
and the trace operator
\begin{equation}
\label{eq:trace}
 \tr: \HL(y^\alpha,\C) \rightarrow \Hs, \qquad w \mapsto \tr w,
\end{equation}
where $\tr w$ denotes the trace of $w$ onto $\Omega \times \{ 0 \}$.

The results of \cite{CT:10,CS:07,CDDS:11,ST:10} thus read as follows: Let $\ue \in \HL(y^\alpha,\C)$ and $\usf \in \Hs$ be the solutions to \eqref{eq:extension} and \eqref{eq:laps}, respectively, then 
\begin{equation}
\label{eq:CSextension}
\usf = \tr \ue.
\end{equation}

A first step toward a discretization scheme is to truncate, for a given truncation parameter $\Y >0$, the semi--infinite cylinder $\C$ to $\C_{\Y}:= \Omega \times (0,\Y)$ and seek solutions in this bounded domain. In fact, let $v \in \HL(y^\alpha,\C_{\Y})$ be the solution to
\begin{equation}
 \label{eq:truncation}
 \int_{\C_{\Y}} y^\alpha \nabla v \cdot \nabla \phi = d_s \langle \fsf, \tr \phi \rangle \qquad \forall \phi \in \HL(y^{\alpha},\C_{\Y}),
\end{equation}
where $\HL(y^{\alpha},\C_{\Y}) = \left\{ w \in H^1(y^\alpha,\C_{\Y}): w = 0 \textrm{ on } \partial_L \C_{\Y} \cup \Omega \times \{\Y\} \right\}$. Then the exponential decay of $\ue$ in the extended variable $y$ implies the following error estimate
\[
 \| \nabla(\ue - v ) \|_{L^2(y^{\alpha},\C)} \lesssim e^{-\sqrt{\lambda_1} \Y/4 } \| \fsf \|_{\Hsd},
\]
provided $\Y \geq 1$, and the hidden constant depends on $s$, but is bounded on compact subsets of $(0,1)$. We refer the reader to \cite[Section 3]{NOS} for details. With this truncation at hand, we thus recall the finite element discretization techniques of \cite[Section 4]{NOS}. 

To avoid technical difficulties, we assume that $\Omega$ is a convex polytopal subset of $\R^n$ and refer the reader to \cite{Otarola_controlp1} for results involving curved domains. Let $\T_\Omega = \{K\}$ be a conforming and shape regular triangulation of $\Omega$ into cells $K$ that are isoparametrically equivalent to either a simplex or a cube. Let $\I_\Y = \{I\}$ be a partition of the interval $[0,\Y]$ with mesh points
\begin{equation}
\label{eq:gradedmesh}
  y_j = \left( \frac{j}M \right)^\gamma \Y, \quad j = 0,\ldots,M, \quad \gamma > \frac3{1-\alpha}=\frac3{2s}>1.
\end{equation}
We then construct a mesh of the cylinder $\C_\Y$ by $\T_\Y = \T_\Omega \otimes \I_\Y$, \ie each cell $T \in \T_\Y$ is of the form $T = K \times I$ where $K \in \T_\Omega$ and $I \in \I_\Y$. We note that, by construction, $\# \T_\Y = M \#\T_\Omega$. When $\T_\Omega$ is quasiuniform with $\# \T_\Omega \approx M^n$ we have $\# \T_\Y \approx M^{n+1}$ and, if $h_{\T_\Omega} = \max \{ \diam(K) : K \in \T_\Omega \}$, then $M \approx h_{\T_\Omega}^{-1}$. Having constructed the mesh $\T_\Y$ we define the finite element space
\begin{equation*}
\label{eq:defofFE}
  \polV(\T_\Y) := \left\{ W \in C^0(\bar\C_\Y): W_{|T} \in \calP(K) \otimes \polP_1(I) \ \forall T \in \T_\Y, \ W_{|\Gamma_D} =0 \right\},
\end{equation*}
where, $\Gamma_D = \partial \Omega \times [0,\Y) \cup \Omega \times \{ \Y \}$, and if $K$ is isoparametrically equivalent to a simplex, $\mathcal{P}(K)=\polP_1(K)$ \ie the set of polynomials of degree at most one. If $K$ is a cube $\calP(K) = \mathbb{Q}_1(K)$, that is, the set of polynomials of degree at most one in each variable. We must immediately comment that, owing to \eqref{eq:gradedmesh}, the meshes $\T_\Y$ are not shape regular but satisfy: if $T_1 = K_1 \times I_1$ and $T_2 = K_2 \times I_2$ are neighbors, then there is $\kappa>0$ such that
\[
 h_{I_1} \leq \kappa h_{I_2}, \qquad h_I = |I|.
\]
The use of anisotropic meshes in the extended direction $y$ is imperative if one wishes to obtain a quasi-optimal approximation error since $\ue$, the solution to \eqref{eq:extension}, possesses a singularity as $y\downarrow0$; see \cite[Theorem 2.7]{NOS}. 

We thus define a finite element approximation of the solution to the truncated problem \eqref{eq:truncation}: Find $V_{\T_{\Y}} \in \polV(\T_\Y)$ such that
\begin{equation}
 \label{eq:V_discrete}
 \int_{\C_{\Y}} y^{\alpha} \nabla V_{\T_{\Y}} \cdot \nabla W = d_s \langle \fsf, \tr W \rangle \quad \forall W \in \polV(\T_\Y).
\end{equation}
With this discrete function at hand, and on the basis of the localization results of Caffarelli and Silvestre, we define an approximation $U_{\T_{\Omega}} \in \mathbb{U}(\T_{\Omega}) = \tr \polV(\T_\Y)$ of the solution $\usf$ to problem \eqref{eq:laps} as follows:
\begin{equation}
 \label{eq:U_discrete}
 U_{\T_{\Omega}}:= \tr V_{\T_{\Y}}.
\end{equation}

\subsubsection{A fully discrete scheme for the fractional identification problem}\label{s:fds}
Following the discussion in \cite{NOS} one observes that many of the stability and error estimates in this work contain constants that depend on $s$. While these remain bounded in compact subsets of $(0,1)$ many of these degenerate or blow up as $s \downarrow 0$ or $s \uparrow 1$. In fact, it is not clear if the PDE \eqref{eq:extension} is well under the passage of these limits. Even if this problem made sense, the Caffarelli-Silvestre extension property \eqref{eq:CSextension} does not hold as we take the limits mentioned above. For this reason, we continue to work under Assumption~\ref{ass:alphabeta}.
We begin by defining the discrete control to state map $S_{\T}$ as follows:
\[
 S_{\T}: (a,b) \rightarrow \mathbb{U}(\T_{\Omega}), \quad s \mapsto U_{\T_{\Omega}}(s),
\]
where $U_{\T_{\Omega}}(s)$ is defined as in \eqref{eq:U_discrete}. We also define the function $j_{\sigma,\T}:(a,b) \rightarrow \R$ as
\begin{equation}
\label{def:jtau}
 j_{\sigma,\T}(s) = \left(U_{\T_{\Omega}}(s) - \usf_d, d_{\sigma} U_{\T_{\Omega}} (s) \right)_{L^2(\Omega)} + \varphi'(s),
\end{equation}
where the centered difference $d_{\sigma}$ is defined as in \eqref{eq:dsigma}. With these elements at hand, we thus define a fully discrete approximation of the optimal identification parameter $\bar s$ as the solution to the following problem: Find $s_{\sigma,\T} \in (a,b)$ such that
\begin{equation}
\label{def:stau}
 j_{\sigma,\T}(s_{\sigma,\T}) = 0.
\end{equation}

We notice that, under the assumption that the map $S_{\T}$ is continuous in $(a,b)$,
the same arguments developed in the proof of Lemma \ref{lm:zeroone} yield the existence of $s_{r,\T}$ and $s_{l,\T}$ in $(a,b)$ such that $j_{\sigma,\T}(s_{r,\T})<0$ and $j_{\sigma,\T}(s_{l,\T}) > 0$. This implies that, if in the bisection algorithm of section \ref{subsec:bisection} we replace $j_{\sigma}$ by $j_{\sigma,\T}$, the step \textbf{Root isolation} can be performed. Consequently, we deduce the convergence of the bisection algorithm and thus the existence of a solution $s_{\sigma,\T} \in (a,b)$ to problem \eqref{def:stau}.

It is then necessary to study the continuity of $S_\T$, but this can be easily achieved because we are in finite dimensions and the problem is linear.

\begin{proposition}[continuity of $S_\T$]
\label{prop:STcont}
For every mesh $\T_{\Y}$, defined as in Section \ref{subsec:NOS}, the map $S_\T$ is continuous on $(a,b)$.
\end{proposition}
\begin{proof}
Let $\{s_k\}_{k\in \mathbb N} \subset (a,b)$ be such that $s_k \to s \in (a,b)$. Since the operator $\tr$, defined as in \eqref{eq:trace}, is continuous \cite[Proposition 2.5]{NOS}, it suffices to show that the application $s \mapsto V_{\T_\Y}(s)$ is continuous. Consider
\begin{equation*}
\label{eq:stoV}
V_{\T_\Y}(s) \in \V(\T_{\Y}): \quad
  \int_{\C_\Y} y^{1-2s} \nabla V_{\T_\Y}(s) \cdot \nabla W_s = d_s \langle \fsf, \tr W_s \rangle
  \quad \forall W_s \in \V(\T_{\Y}),
\end{equation*}
and
\begin{equation*}
\label{eq:sktoVk}
V_{\T_\Y}(s_k) \in \V(\T_{\Y}): \quad
  \int_{\C_\Y} y^{1-2s_k} \nabla V_{\T_\Y}(s_k) \cdot \nabla W_k = d_{s_k} \langle \fsf, \tr W_k \rangle \quad \forall W_k \in \V(\T_{\Y}).
\end{equation*}
Set $W_s = V_{\T_\Y}(s) - V_{\T_\Y}(s_k)$ and $W_k = V_{\T_\Y}(s_k) - V_{\T_\Y}(s)$ and add these two identities to obtain
\begin{multline*}
  \| \nabla( V_{\T_\Y}(s) - V_{\T_\Y}(s_k) ) \|_{L^2(y^{1-2s},\C_\Y)}^2 = (d_s - d_{s_k}) \langle \fsf, \tr(V_{\T_\Y}(s) - V_{\T_\Y}(s_k)) \rangle \\ +
  \int_{\C_\Y} (y^{1-2s_k}-y^{1-2s}) \nabla V_{\T_\Y}(s_k) \cdot \nabla (V_{\T_\Y}(s) - V_{\T_\Y}(s_k)) = \textrm{I} + \textrm{II}.
\end{multline*}
We now proceed to estimate each one of these terms. 

For the first term we have
\[
  |\textrm{I}| \leq |d_s - d_{s_k}| \|\fsf\|_{L^2(\Omega)} \| \tr(V_{\T_\Y}(s) - V_{\T_\Y}(s_k) \|_{L^2(\Omega)} \to 0
\]
as $k \to \infty$. This is the case because $\| \tr(V_{\T_\Y}(s) - V_{\T_\Y}(s_k) \|_{L^2(\Omega)}$ is uniformly bounded \cite[Proposition 2.5]{NOS} and, by Assumption~\ref{ass:alphabeta}, we have that $d_{s_k} \to d_s$.

We estimate the second term as follows
\[
  |\textrm{II}|  \leq |\Omega| \| \nabla V_{\T_\Y}(s_k) \|_{L^\infty(\C_\Y)} \| \nabla (V_{\T_\Y}(s) - V_{\T_\Y}(s_k)) \|_{L^\infty(\C_\Y)}  
  \int_0^\Y | y^{1-2s} - y^{1-2s_k}|.
\]
Using that we are in finite dimensions, the question reduces to the convergence
\[
  \int_0^\Y | y^{1-2s} - y^{1-2s_k}| \to 0,
\]
which follows from the a.e. convergence of $y^{1-2s_k}$ to $y^{1-2s}$, the fact that, for $0<y < 1$, we have $0<y^{1-2s_k} \leq y^{1-2a} \in L^1(0,1)$ and an application of the dominated convergence theorem.

This concludes the proof.
\end{proof}

We now proceed to derive an a priori error bound for the error between the exact identification parameter $\bar s$ and its approximation $s_{\sigma,\T}$ given as the solution \eqref{def:stau}. We begin by noticing that, following the proof of Lemma~\ref{lem:jsigmatoj}, using \cite[Proposition 28]{NOS3} and Assumption~\ref{ass:alphabeta} we have
\begin{equation}
\label{eq:jsTtojsigma}
  | j_\sigma(s) - j_{\sigma, \T}(s) | \lesssim \frac1\sigma |\log(\#\T_\Y)|^{2b} (\#\T_\Y)^{-(1+a)/(n+1)},
\end{equation}
where the hidden constant depends on $a$ and $b$ but is uniform in $(a,b)$. Clearly, for fixed $\sigma$, this implies the uniform convergence of $j_{\sigma,\T}$ to $j_\sigma$ as we refine the mesh. By repeating the arguments of Lemma~\ref{lem:ssigmatos} we conclude the convergence, up to subsequences, of $\{s_{\sigma,\T}\}_{\T}$ to $s_\sigma$, a root of $j_\sigma$. Arguing as in Remark~\ref{rem:notfullseq}, we see that we cannot expect convergence of the entire family.

Finally, we denote one of these convergent subsequences by $\{s_{\sigma,\T}\}_{\T}$ and provide an error estimate.

\begin{theorem}[Error estimate: discretization in $s$ and space]
 Let $\bar s$ be optimal for the identification problem \eqref{eq:minJ}--\eqref{eq:laps} and $s_{\sigma,\T}$ its approximation defined as the solution to \eqref{def:stau}. If $\sigma$ is sufficiently small, $\#\T_\Y$ is sufficiently large and, 
 $\fsf \in \mathbb{H}^{1-a}(\Omega)$, then 
 \begin{equation}
  \label{eq:error_estimate_1}
   | \bar s -s_{\sigma,\T} | \lesssim \sigma^{-1} | \log (\# \T_{\Y})|^{2b} (\# \T_{\Y})^{-(1+a)/(n+1)} \| \fsf \|_{\mathbb{H}^{1-a}(\Omega)} + \sigma^2,
 \end{equation}
 where the hidden constant is independent of $\bar s$, $s_{\sigma,\T}$, $\fsf$ and the mesh $\T_{\Y}$.
\end{theorem}
\begin{proof}
We begin by remarking that, by setting $\sigma$ sufficiently small and $\#\T_\Y$ sufficiently large, respectively, we can assert that $s_{\sigma,\T} \in (\bar{s} - \delta, \bar{s} + \delta)$ with $\delta$ being the parameter of Corollary~\ref{co:local_convexity}. By invoking the estimate \eqref{eq:local_convexity} and in view of the fact that $f'(\bar s) = 0 = j_{\sigma,\T}(s_{\sigma,\T})$, we deduce the following estimate:
\begin{multline*}
\frac \vartheta2| \bar s -s_{\sigma,\T} |^2  \leq \left( f'(\bar s) - f'(s_{\sigma,\T}) \right) \cdot (\bar s - s_{\sigma,\T} ) = \left( j_{\sigma,\T}(s_{\sigma,\T})-f'(s_{\sigma,\T} )  \right) \cdot (\bar s - s_{\sigma,\T} ). 
\end{multline*}

We proceed to bound the right hand side of the previous expression. To accomplish this task, we invoke the definition \eqref{def:jtau} of $j_{\sigma,\T}$ and repeating the arguments of Lemma~\ref{lem:jsigmatoj} we obtain that
\begin{equation}
\label{I+II}
  \begin{aligned}
    |j_{\sigma,\T}(s_{\sigma,\T})-f'(s_{\sigma,\T})| &\leq \left|\left(U_{\T_{\Omega}}(s_{\sigma,\T}) - \usf_d, d_{\sigma} U_{\T_{\Omega}} (s_{\sigma,\T}) - D_s \usf(s_{\sigma,\T}) \right)_{L^2(\Omega)} \right| 
    \\
    &+ \left|\left( U_{\T_{\Omega}}(s_{\sigma,\T}) - \usf(s_{\sigma,\T}), D_s \usf(s_{\sigma,\T}) \right)_{L^2(\Omega)}\right| = \textrm{I} + \textrm{II}.
  \end{aligned}
\end{equation}
We thus examine each term separately. We start with $\textrm{II}$: its control relies on the a priori error estimates of \cite{NOS,NOS3}. In fact, combining the results of \cite[Proposition 28]{NOS3} with the estimate \eqref{eq:Dscontrol} for $m=1$, we arrive at
\begin{align*}
  | \textrm{II} | &\leq \| D_s \usf(s_{\sigma,\T}) \|_{L^2(\Omega)} \|  U_{\T_{\Omega}}(s_{\sigma,\T}) - \usf(s_{\sigma,\T}) \|_{L^2(\Omega)} \\
    &\lesssim s_{\sigma,\T}^{-1}| \log (\# \T_{\Y})|^{2 s_{\sigma,\T}} (\# \T_{\Y})^{-(1+s_{\sigma,\T})/(n+1)} \| \fsf \|_{\mathbb{H}^{1-s_{\sigma,\T}}(\Omega)} \\
    & \lesssim | \log (\# \T_{\Y})|^{2b} (\# \T_{\Y})^{-(1+a)/(n+1)} \| \fsf \|_{\mathbb{H}^{1-a}(\Omega)}
\end{align*}
where the hidden constant depends on $a$ and $b$ but is independent of $\bar s$, $s_{\sigma,\T}$, $\fsf$ and $\T_{\Y}$. Notice that here we used Assumption~\ref{ass:alphabeta} to, for instance, control the term $s_{\sigma,\T}^{-1}$.

We now proceed to control the term $\textrm{I}$ in \eqref{I+II}. A basic application of the Cauchy--Schwarz inequality yields
\[
 | \textrm{I} | \leq \| U_{\T_{\Omega}}(s_{\sigma,\T}) - \usf_d \|_{L^2(\Omega)} \| d_{\sigma} U_{\T_{\Omega}} (s_{\sigma,\T}) - D_s \usf(s_{\sigma,\T}) \|_{L^2(\Omega)}.
\]
We thus apply the estimate \eqref{eq:Sbbindep} and the triangle inequality to obtain that
\[
 | \textrm{I} | \lesssim \| d_{\sigma} \left( U_{\T_{\Omega}} (s_{\sigma,\T}) - \usf(s_{\sigma,\T}) \right)\|_{L^2(\Omega)}
 + \| d_{\sigma} \usf(s_{\sigma,\T}) - D_s \usf (s_{\sigma,\T}) \|_{L^2(\Omega)}.
\]
We estimate the first term on the right hand side of the previous expression: the definition \eqref{eq:dsigma} of $d_{\sigma}$ and \cite[Proposition 28]{NOS3} imply that 
\begin{multline*}
 \| d_{\sigma} \left( U_{\T_{\Omega}} (s_{\sigma,\T}) - \usf(s_{\sigma,\T} ) \right)\|_{L^2(\Omega)} \leq \frac1{2\sigma}\Big( \| U_{\T_{\Omega}} (s_{\sigma,\T} + \sigma) - \usf(s_{\sigma,\T} + \sigma) \|_{L^2(\Omega)}
 \\
 + \| U_{\T_{\Omega}} (s_{\sigma,\T} - \sigma) - \usf(s_{\sigma,\T} - \sigma) \|_{L^2(\Omega)} \Big) \lesssim \frac{1}{\sigma}|\log (\# \T_{\Y})|^{2b} (\# \T_{\Y})^{-\frac{1+a}{n+1}} \| \fsf \|_{\mathbb{H}^{1-a}(\Omega)};
\end{multline*}
we notice that $\sigma$ is small enough such that $s_{\sigma,\T} \pm \sigma \in (a,b)$. On the other hand, an estimate similar to \eqref{eq:bouncds} yields that
\begin{equation*}
  \| D_s \usf(s_{\sigma,\T}) - d_{\sigma} \usf(s_{\sigma,\T}) \|_{L^2(\Omega)} \lesssim \sigma^2 a^{-3}. 
\end{equation*}
Collecting the previous estimates we arrive at the following bound for the term $\textrm{I}$:
\begin{equation}
\label{eq:I}
|\textrm{I}| \lesssim \sigma^{-1}| \log (\# \T_{\Y})|^{2b} (\# \T_{\Y})^{-(1+a)/(n+1)} \| \fsf \|_{\mathbb{H}^{1-a}(\Omega)} + \sigma^2 a^{-3}.
\end{equation}
On the basis of \eqref{I+II}, this bound, and the estimate for the term $\textrm{II}$ yield
\[
 | \bar s -s_{\sigma,\T} | \lesssim \sigma^{-1} | \log (\# \T_{\Y})|^{2b} (\# \T_{\Y})^{-(1+a)/(n+1)} \| \fsf \|_{\mathbb{H}^{1-a}(\Omega)} + \sigma^2,
\]
where the hidden constant depends on $a$ and $b$, but is independent of $\sigma$ and $\#\T_\Y$.
This concludes the proof.
\end{proof}

A natural choice of $\sigma$ comes from equilibrating the terms on the right--hand side of \eqref{eq:error_estimate_1}: $\sigma \approx |\log(\#\T_\Y)|^{2b/3} (\# \T_{\Y})^{-(1+a)/3(n+1)}$. This implies the following error estimate.

\begin{corollary}[error estimate: discretization in $s$ and space]
\label{cor:final_est}
Let $\bar s$ be optimal for the identification problem \eqref{eq:minJ}--\eqref{eq:laps} and $s_{\sigma,\T}$ be its approximation defined as the solution to \eqref{def:stau}. If $\#\T_\Y$ is sufficiently large, the parameter $\sigma$ is chosen as
\[
  \sigma \approx |\log(\#\T_\Y)|^{2b/3} (\# \T_{\Y})^{-(1+a)/3(n+1)},
\]
and 
$\fsf \in \mathbb{H}^{1-a}(\Omega)$ then
\begin{equation}
\label{eq:error_estimate_2}
  | \bar s-s_{\sigma,\T} |  \lesssim | \log (\# \T_{\Y})|^{4b/3} (\# \T_{\Y})^{-\frac{2(1+a)}{3(n+1)}},
\end{equation}
where the hidden constant depends on $a$ and $b$ but is independent of $\bar s$, $s_{\sigma,\T}$, and the mesh $\T_{\Y}$.
\end{corollary}

\section{Numerical examples}\label{s:nex}
In this section, we study the performance of the proposed bisection algorithm of section~\ref{sec:Numerics} when applied to the fully discrete parameter identification problem of  section~\ref{s:fds} with the help of four numerical examples. 

The implementation has been carried out within the MATLAB software library $i$FEM \cite{chen2009ifem}. The stiffness matrices of the discrete system \eqref{eq:V_discrete} are assembled exactly and the forcing terms are computed by a quadrature rule which is exact for polynomials up to degree 4. Additionally, the first term in \eqref{def:jtau} is computed by a quadrature formula which is exact for polynomials of degree 7. All the linear systems are solved exactly using MATLAB's built-in direct solver. 

In all examples, $n=2$, $\Omega = (0,1)^2$, $\textsf{TOL} = 2.2204e$-16, and the initial value of $s_l$, $s_r$ is 0.3, and 0.9, respectively. The truncation parameter for the cylinder $\C_{\Y}$ is $\Y = 1+\frac{1}{3}(\#\T_\Omega)$ which allows balancing the approximation and truncation errors for our state equation, see \cite[Remark~5.5]{NOS}. Moreover, 
\[
 \sigma = \frac{1}{2.5}(\# \T_{\Y})^{-\frac{(1+\epsilon)}{9}},
\]
with $\epsilon = 10^{-10}$.

Under the above setting, the eigenvalues and eigenvectors of $-\Delta$ are:
\[
 \lambda_{k,l} = \pi^2(k^2+l^2), \quad \varphi_{k,l}(x_1,x_2) = \sin(k\pi x_1) \sin(l\pi x_2), \quad 
 k, l \in \mathbb{N} . 
\]
Consequently, by letting $\fsf = \lambda^s_{2,2}\varphi_{2,2}$ for any $s\in(0,1)$ we obtain $\ousf = \varphi_{2,2}$.

In what follows we will consider four examples. In the first one we set $\bar s=1/2$, $\fsf$ and $\ousf$ as above and we set $\usf_d = \ousf$. The second one differs from the first one in that we set $\bar s = (3-\sqrt{5})/2$. In our third example, the exact solution is not known. 
Finally, in our last example we explore the robustness of our algorithm with respect to perturbations in the data. We accomplish this by considering the same setting as in the first example but we add a random perturbation $r \in (-e,e)$ to the right hand side $\fsf$. We then explore the behavior of the optimal parameter $\bar s$ as the size of the perturbation $e$ varies.

\subsection{Example 1}\label{s:ex1}
We recall the definition of the cost function $J(\usf,s)$ from \eqref{def:J} and set $\varphi(s) = \frac{1}{s(1-s)}$. The latter is strictly convex over the interval $(0,1)$ and fulfills the conditions in \eqref{eq:cond_varphi}. The optimal solution $\bar{s}$ to \eqref{eq:minJ}--\eqref{eq:laps} is given by $\bar{s} = 1/2$. 

Table~\ref{t:ex1} illustrates the performance of our optimization solver. The first column indicates the degrees of freedom $\# \T_{\Y}$, the second column shows the value of $s_{\sigma,\T}$ obtained by solving \eqref{def:stau}, and the third column shows the corresponding value $j_{\sigma,\T}$ at $s_{\sigma,\T}$. The final column shows the total number of optimization iterations $N$ taken, for the bisection algorithm to converge. We notice that the observed values of $s_{\sigma,\T}$ matches almost perfectly with $\bar{s}$. In addition, the pattern in $N$, as we refine the mesh, indicates a mesh-independent behavior.

\begin{table}[htb]
\centering
\begin{tabular}{|r|c|c|c|} \hline
$\# \T_{\Y}$ &    $s_{\sigma,\T}$    &    $j_{\sigma,\T}(s_{\sigma,\T})$ & $N$  \\ \hline
 3146 &  4.96572e-01  & -8.89011e-14 &  53   \\
10496 &  4.98371e-01  & -8.38218e-14 &  53   \\
25137 &  4.99069e-01  &  3.49235e-14 &  53   \\
49348 &  4.99402e-01  &  1.52327e-12 &  53   \\ 
85529 &  4.99585e-01  &  6.28221e-12 &  53   \\ \hline 
\end{tabular}
\caption{\label{t:ex1}The first column indicates the degrees of freedom, the second one corresponds to the solution $s_{\sigma,\T}$ of our discrete optimality system \eqref{def:stau} and the third column illustrates the corresponding value of $j_{\sigma,\T}$ at $s_{\sigma,\T}$. The final column shows, $N$, the number of iterations taken by the bisection algorithm to converge. The values of $N$ are moderate. Additionally, we observe that $s_{\sigma,\T}$ matches with the exact solution $\bar{s} = 1/2$ and the pattern in $N$ shows a mesh independent behavior upon mesh refinement.}
\end{table}

Figure~\ref{f:ex1_2} (left panel) shows the computational rate of convergence. We observe that
\[
 |\bar{s}-s_{\sigma,\T}| \lesssim (\#\T_\Y)^{-0.6}
\]
which is significantly better than the predicated rate of $(\#\T_\Y)^{-0.22}$ by the Corollary~\ref{cor:final_est}. Indeed this suggests that our theoretical rates are pessimistic and in practice, our algorithm works much better.

\begin{figure}[htb]
\includegraphics[width=0.48\textwidth]{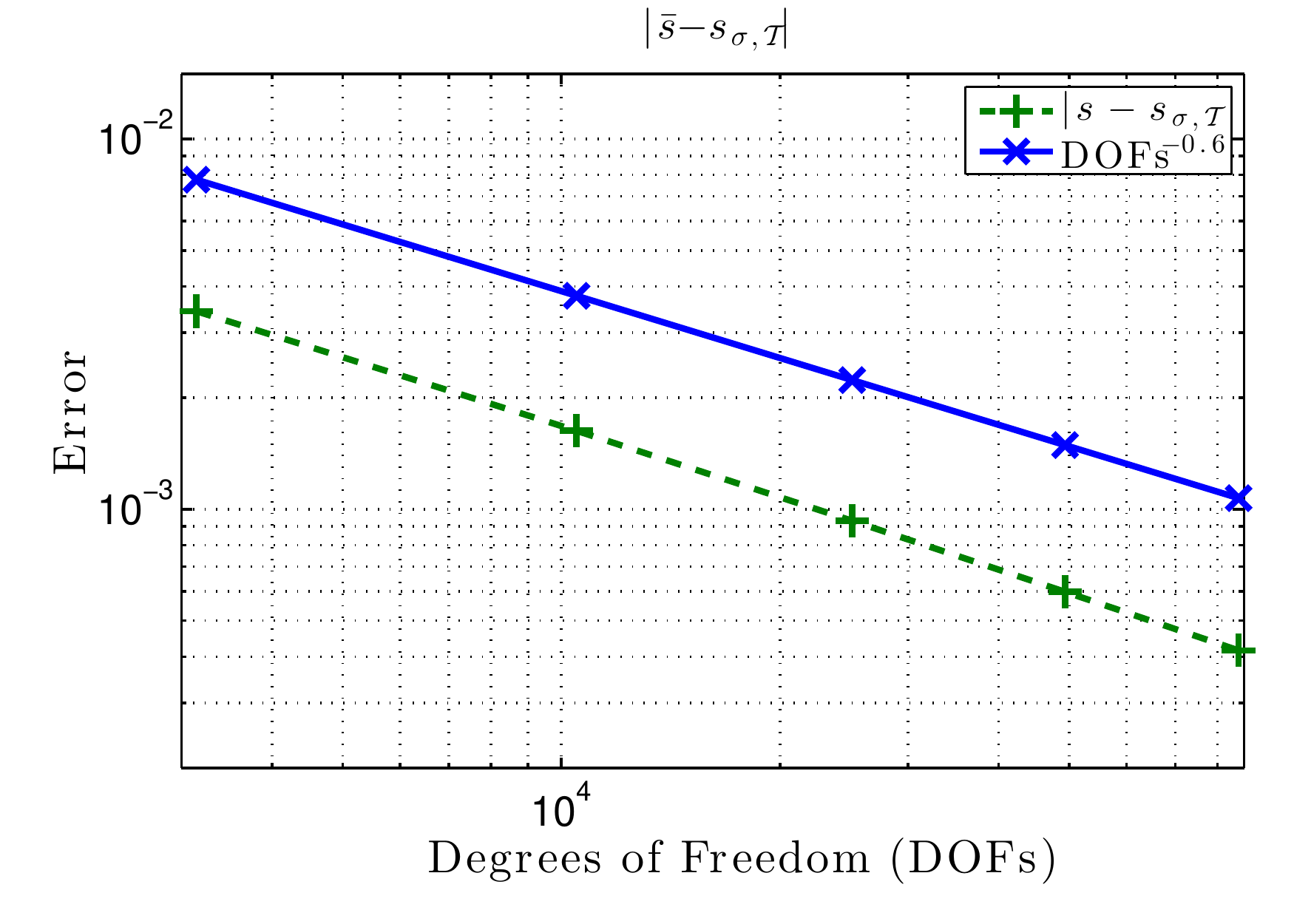} \quad 
\includegraphics[width=0.48\textwidth]{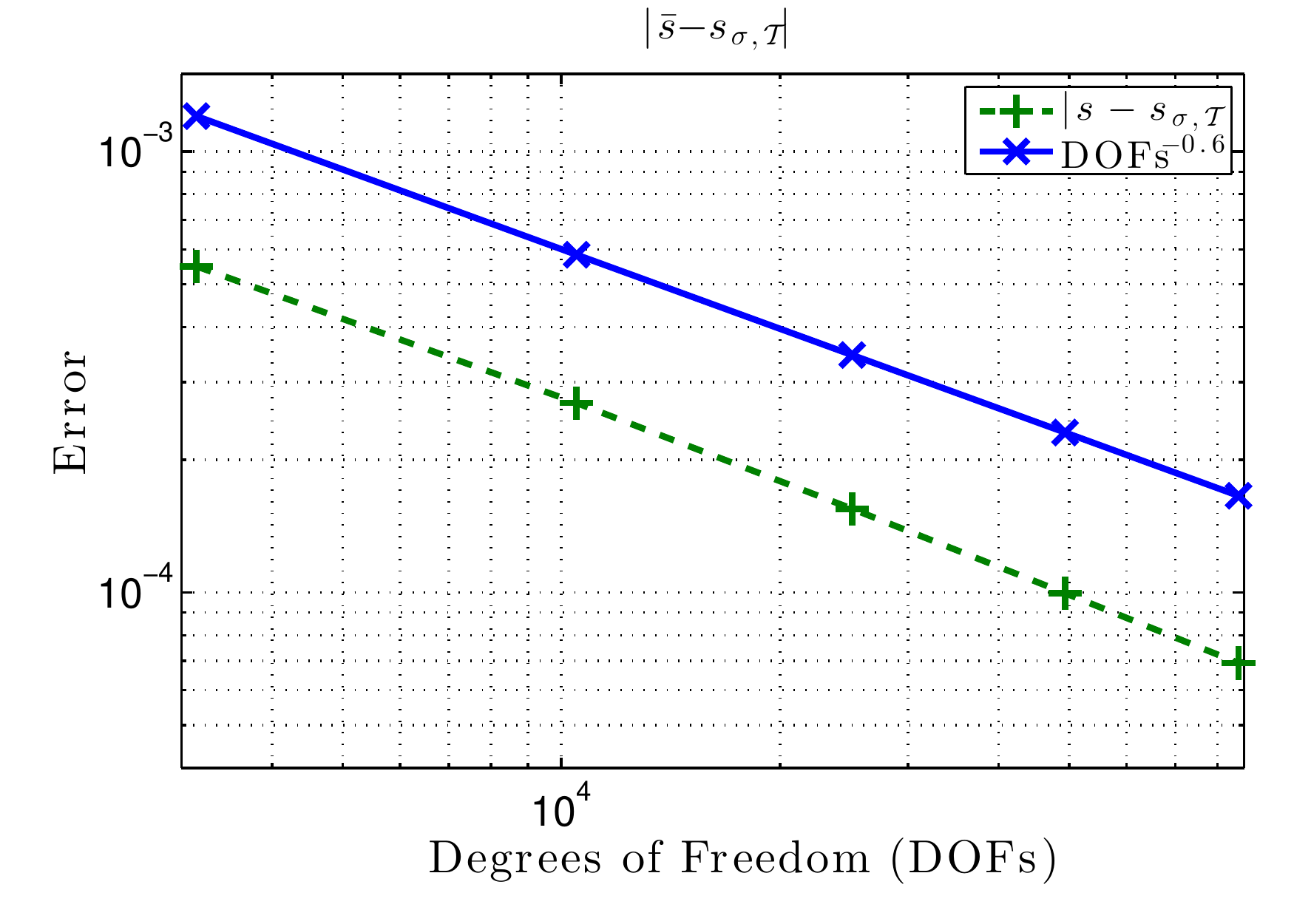}
\caption{\label{f:ex1_2}The left panel (dotted curve) shows the convergence rate for Example~1 and the right one for Example~2. The solid line is the reference line. We notice that the computational rates of convergence, in both examples, are much higher than the theoretically predicted rates in Corollary~\ref{cor:final_est}.}
\end{figure}

\subsection{Example 2}\label{s:ex2}
We set $\varphi(s) = s^{-1} e^{\frac{1}{(1-s)}}$ which is again strictly convex over the interval $(0,1)$ and fulfills the conditions in \eqref{eq:cond_varphi}. The optimal solution $\bar{s}$ to \eqref{eq:minJ}--\eqref{eq:laps} is given by $\bar{s} = (3-\sqrt{5})/2$. 

Table~\ref{t:ex2} illustrates the performance of our optimization solver. As we noted in section~\ref{s:ex1}, the numerically computed solution $s_{\sigma,\T}$ matches almost perfectly with $\bar{s}$ and the pattern of $N$, with mesh refinement, again indicates a mesh independent behavior. 

\begin{table}[htb]
\centering
\begin{tabular}{|r|c|c|c|} \hline
$\# \T_{\Y}$ &    $s_{\sigma,\T}$    &    $j_{\sigma,\T}(s_{\sigma,\T})$ & $N$  \\ \hline
 3146 &  3.81417e-01  &  9.99201e-16 &  46    \\
10496 &  3.81697e-01  & -2.52812e-13 &  53    \\
25137 &  3.81811e-01  &  1.36418e-12 &  53    \\
49348 &  3.81866e-01  &  2.66251e-12 &  53    \\ 
85529 &  3.81897e-01  &  3.53083e-12 &  53 \\ \hline 
\end{tabular}
\caption{\label{t:ex2}The first column indicates the degrees of freedom, the second one corresponds to the solution $s_{\sigma,\T}$ of our discrete optimality system \eqref{def:stau} and the third column illustrates the corresponding value of $j_{\sigma,\T}$ at $s_{\sigma,\T}$. The final column shows, $N$, the number of iterations taken by the bisection algorithm to converge. The values of $N$ are moderate. Additionally, we observe that $s_{\sigma,\T}$ matches with the exact solution $\bar{s} = (3-\sqrt{5})/2$ and the pattern in $N$ shows a mesh independent behavior upon mesh refinement.}
\end{table}

Figure~\ref{f:ex1_2} (right panel) shows the computational rate of convergence. We again see that
\[
 |\bar{s}-s_{\sigma,\T}| \lesssim (\#\T_\Y)^{-0.6}
\]
Thus the observed rate is far superior than the theoretically predicted rate in Corollary~\ref{cor:final_est}.

\subsection{Example 3}

In our third example, we take $\varphi(s) = s^{-1} e^{\frac{1}{(1-s)}}$, $\fsf = 10$, and $\usf_d = \max \big\{0.5-\sqrt{|x_1-0.5|^2+|x_2-0.5|^2}, 0 \big\}$. We notice that $\fsf$ is large, thus the requirements of Theorem~\ref{thm:rate_s} are not necessarily fulfilled. In addition, for $\mu\le 1/2$, $\fsf \not\in \mathbb{H}^{1-\mu}(\Omega)$ thus the requirements of Corollary~\ref{cor:final_est} are not fulfilled. Nevertheless, as we illustrate in Table~\ref{t:ex3}, we can still solve the problem. We again notice a mesh independent behavior in the number of iterations ($N$) taken by the bisection algorithm to converge.

\begin{table}[htb]
\centering
\begin{tabular}{|r|c|c|c|} \hline
$\# \T_{\Y}$  &      $s_{\sigma,\T}$    &    $j_{\sigma,\T}(s_{\sigma,\T})$ & $N$  \\ \hline 
3146  &  4.44005e-01  &  4.22951e-12 &  53 \\
10496 &  4.47239e-01  &  2.97451e-11 &  53 \\
25137 &  4.48182e-01  & -3.20792e-11 &  53 \\
49348 &  4.48544e-01  &  4.83542e-11 &  53 \\
85529 &  4.48690e-01  &  2.68390e-10 &  53 \\ \hline 
\end{tabular}
\caption{\label{t:ex3}The first column indicates the degrees of freedom, the second one corresponds to the solution $s_{\sigma,\T}$ of our discrete optimality system \eqref{def:stau} and the third column illustrates the corresponding value $j_{\sigma,\T}$ at $s_{\sigma,\T}$. The final column shows, $N$, the number of iterations taken by the bisection algorithm to converge. The values of $N$ are moderate and show a mesh independent character.}
\end{table}

\subsection{Example 4}

In our final example we consider a similar setup to subsection~\ref{s:ex1}. We modify the right hand side $\fsf = \lambda_{2,2}^{\bar{s}} \sin(2\pi x_1) \sin(2\pi x_2)$, with $\bar{s}=1/2$, by adding a uniformly distributed random parameter $r \in (-e,e)$. We fix the spatial mesh to $\#\T_\Y = 85,529$. 

At first we set $e = 200$, as a result $r$ is more than 200 times the actual signal $\fsf$, see the first row on Table~\ref{t:ex4}. Despite such a large noise, the recovery of $\bar{s}$ is reasonable. Letting $e \downarrow 0$, we can recover $\bar{s}$ almost perfectly.

\begin{table}[htb]
\centering
\begin{tabular}{|r|c|c|c|} \hline
$e$   &      $s_{\sigma,\T}$    &    $j_{\sigma,\T}(s_{\sigma,\T})$ & $N$  \\ \hline 
200 & 6.33937e-01  & 7.28484e-12 & 53 \\
20 & 5.06469e-01  & -5.17408e-12 & 53 \\
2  & 4.99341e-01  & -7.37949e-12 & 53 \\
0.5 & 4.99581e-01  & -5.68941e-12 & 53 \\
0.25 & 4.99586e-01  & 3.64379e-12 & 53 \\
0.125 & 4.99584e-01  & 3.33318e-13 & 53 \\ \hline 
\end{tabular}
\caption{\label{t:ex4}
Robustness of our algorithm with respect to noisy data. The number of spatial degrees of freedom is fixed to $\# \T_{\Y} = 85,529$. 
The first column indicates the range of the uniformly distributed parameter $r$ which is added to the right hand side $\fsf$, the second one corresponds to the solution $s_{\sigma,\T}$ of our discrete optimality system \eqref{def:stau} and the third column illustrates the corresponding value $j_{\sigma,\T}$ at $s_{\sigma,\T}$. The final column shows $N$, the number of iterations taken by the bisection algorithm to converge. Notice that even with a noise which is 200 times more than the actual signal $\fsf$ the recovery of $\bar{s}$ is reasonable (first row). If the noise is of the same order as $\fsf$ we can recover $\bar{s}$ perfectly. The values of $N$ are moderate and show a mesh independent character.}
\end{table}

\bibliographystyle{plain}
\bibliography{biblio}
\end{document}